\documentclass[reqno]{amsart}
\usepackage{amsmath,amsfonts,amssymb,amscd,verbatim,multicol}
\usepackage{fullpage}

\usepackage{graphicx}
\usepackage{stmaryrd}
\usepackage{mathtools}
\usepackage{dsfont}
\usepackage{xcolor}
\usepackage{enumitem}
\usepackage{cancel}
\usepackage[normalem]{ulem}

\numberwithin{equation}{section}
\newtheorem{theorem}{Theorem} 
\numberwithin{theorem}{section}

\newtheorem{corollary}[theorem]{Corollary} 
\newtheorem{lemma}[theorem]{Lemma}
\newtheorem{proposition}[theorem]{Proposition}

\newtheorem{definition}[theorem]{Definition}

\newtheorem{hypothesis}[theorem]{Hypothesis} 

\newtheorem{remark}[theorem]{Remark}

\newtheorem*{theorem*}{Theorem}

\DeclareMathOperator\Aut{Aut}

\DeclareMathOperator\lcm{lcm}

\DeclareMathOperator\ord{ord}

\DeclareMathOperator\supp{supp}

\newcommand\NN{\mathbb N}
\newcommand\ZZ{\mathbb Z}

\newcommand\be{\mathbf e}
\newcommand\bu{\mathbf u}
\newcommand\bv{\mathbf v}

\newcommand\bzero{\mathbf 0}

\newcommand\kk{\Bbbk}

\newcommand\inv{^{-1}}
\newcommand\iso{\cong}

\newcommand\tensor{\otimes}
\newcommand\qder{\delta}
\newcommand\GWAdeg{p}
\newcommand\og{s}
\newcommand\om{m}

\newcommand{\grp}[1]{{\langle {#1} \rangle}}

\begin{document}

\title{Pointed Hopf actions on quantum generalized Weyl algebras}
\author[Gaddis]{Jason Gaddis}
\address{(Gaddis) Miami University, Department of Mathematics, Oxford, Ohio 45056} 
\email{gaddisj@miamioh.edu}

\author[Won]{Robert Won}
\address{(Won) The George Washington University, Department of Mathematics, Washington, DC 20052}
\email{robertwon@gwu.edu}



\begin{abstract}
We study actions of pointed Hopf algebras in the $\ZZ$-graded setting. Our main result classifies inner-faithful actions of generalized Taft algebras on quantum generalized Weyl algebras which respect the $\ZZ$-grading. We also show that generically the invariant rings of Taft actions on quantum generalized Weyl algebras are commutative rings.
\end{abstract}

\maketitle

\section{Introduction}

The goal of this work is to study Hopf actions in the setting of $\ZZ$-graded algebras. The Weyl algebra is an example of such an algebra, but the Weyl algebra has no finite quantum symmetry \cite{CEW2}. Instead, we study generalized Weyl algebras (GWAs), specifically quantum GWAs. By work of Su\'{a}rez-Alvarez and Vivas, the automorphism groups of quantum GWAs are well-understood \cite{SAV}. Here we classify actions of (generalized) Taft algebras on quantum GWAs that respect their $\ZZ$-grading.

\begin{definition}{{\cite{rad2}}}
\label{defn.taft}
Let $m, n \in \NN$ such that $m > 1$ and $m \mid n$, and let $\lambda \in \kk$ be a primitive $m^\text{th}$ root of unity. The \emph{generalized Taft algebra} corresponding to this data is
\begin{align}
\label{eq.taft}
T_n(\lambda,m) := \kk\langle x,g \mid g^n-1,x^m, gx-\lambda xg \rangle.
\end{align}
\end{definition}

Actions of (generalized) Taft algebras on various families of algebras have been considered in a variety of works.
Actions on finite-dimensional algebras with a single generator were studied by Montgomery and Schneider \cite{MSch}. These results were recently generalized by Cline \cite{cline}. Furthermore, Centrone and Yasumura studied actions on general finite-dimensional algebras \cite{CY}, while Bahturin and Montgomery classified actions on matrix algebras \cite{BM}. Kinser and Walton considered Taft algebra actions on path algebras of quivers \cite{KW}. 
Chen, Wang, and Zhang studied actions of nonsemisimple Hopf algebras on Artin--Schelter regular algebras in characteristic $p$ \cite{CWZ2}.
This work provides a unique perspective on Taft algebra actions. Our primary goal is to classify actions by generalized Taft algebras on quantum GWAs. 

In some sense, this is a natural extension of work on Hopf actions on Artin--Schelter regular algebras \cite{CKWZ1,KWZ}. By a result of Liu \cite{liu}, these GWAs are twisted Calabi--Yau, and so this work may be seen as an extension of the study of quantum symmetries to the $\ZZ$-graded setting. 

\begin{definition}
\label{defn.qgwa}
Let $q \in \kk^\times$ and let $h(t) \in \kk[t]$ be non-constant. The \emph{quantum generalized Weyl algebra} corresponding to this data is
\begin{align}
\label{eq.qgwa}
\kk[t](u,v,q,h) = \kk\langle u,v,t \mid ut-qtu, vt-q\inv tv, vu=h(t), uv=h(qt)\rangle.
\end{align}
\end{definition}

There is a natural $\ZZ$-grading on $A=\kk[t](u,v,q,h)$ obtained by setting $\deg t = 0$, $\deg u=1$, and $\deg v = -1$. 
The \emph{quantum planes} 
\[ \kk_q[u,v] = \kk\langle u,v \mid uv-qvu \rangle\]
are quantum GWAs with $h=t$, while the \emph{quantum Weyl algebras} 
\[ A_1^q(\kk) = \kk\langle u,v \mid uv-qvu-1 \rangle\]
are quantum GWAs with $h=t-1$.
In \cite{GWY}, Yee and the authors classified actions of Taft algebras on quantum planes and quantum Weyl algebras that are linear in the sense that $r(u), r(v) \in \kk u + \kk v$ for all $r \in T$. This was extended to linear actions of certain generalized Taft algebras, as well as their higher-dimensional analogues, by Cline and the first author \cite{CG}. 

In this work, we consider actions that respect the $\ZZ$-grading on quantum GWAs. More specifically, if $A$ is a quantum GWA, then we consider actions by generalized Taft algebras $T=T_n(\lambda,m)$ such that $A_0$ and $A_{-i} \oplus A_{i}$ are $T$-modules for every $i \in \mathbb{N}$. 
In this way, we classify actions that are distinct from those obtained in \cite{CG,GWY}. 

We remark that our notion of respecting the $\ZZ$-grading is more general than the usual notion of assuming that each homogeneous component of $A$ is a $T$-submodule. Hence, we call such an action \emph{weakly $\ZZ$-graded}.
The weakly $\ZZ$-graded setting captures (for example) group actions that preserve the $\ZZ$-grading of $A$ up to the automorphism of $\ZZ$ which sends $1$ to $-1$. 
Our main theorem, Theorem~\ref{thm.main}, shows that when $q \neq -1$, every weakly $\ZZ$-graded action is in fact $\ZZ$-graded in the usual sense. When $q = -1$, then Theorem~\ref{thm.omega} shows that there exist actions that are weakly $\ZZ$-graded but not $\ZZ$-graded.

Our classification depends on the relationship between the orders of the various roots of unity, as well as the support of the relevant polynomials. In this paper, the \emph{support of a polynomial} $h=\sum h_i t^i \in \kk[t]$ is defined to be $\supp(h)=\{i \mid h_i \neq 0\} \subset \ZZ$.

\begin{theorem}
\label{thm.intro}
Let $A = \kk[t](u,v,\sigma,h)$ be a quantum GWA with $\sigma(t)=qt$ where $q^2 \neq 1$ and $h$ is a nonconstant polynomial.
Let $T = T_n(\lambda,m)$ be a generalized Taft algebra. 

(A) There is an inner-faithful weakly $\ZZ$-graded $T$-module algebra structure on $A$ if and only if 
\begin{enumerate}
\item 
$\supp(h)$ is contained in a single congruence class modulo $m$, and
\item there exists an integer $k$ coprime to $m$ such that $\lcm(m, \ord(q^{k})) = n$.
\end{enumerate}

(B) Assuming the conditions in (A) are satisfied, the inner-faithful weakly $\ZZ$-graded $T$-module algebra structures on $A$ are parametrized by
$\gamma,\mu \in \kk^\times$ and $\phi(t) \in \kk[t]$ of degree $d$ such that
\begin{enumerate}
\item $\ord(\gamma) = m$ and $\lambda = \gamma^{d-1}$,
\item $\lcm(m, \ord(\mu)) = n$,
\item $\supp(\phi)$ is contained in a single congruence class modulo $m$, and
\item $\mu q^{1-d}$ is an $m^\text{th}$ root of unity.
\end{enumerate}
\end{theorem}


The conditions defined above do guarantee an action even in the case of $q=-1$. However, when $q = -1$, the automorphism group of $A$ is more complicated and so there are additional $T$-actions. These are discussed in Theorem \ref{thm.omega}


Our results can be viewed as a quantum version of the classical problem of determining the groups that act faithfully on an algebra $A$. It is therefore interesting to study which cyclic subgroups $G$ of $\Aut(A)$ are in fact restrictions of an inner-faithful $T$-action to the group of group-like elements. Such a $T$-action can be viewed as a ``quantum thickening" of the action of $G$. 
In order to study this question, we first discuss $\Aut(A)$, which was determined by Su\'{a}rez-Alvarez and Vivas \cite{SAV}.

Let $A=\kk[t](u,v,\sigma,h)$ be a quantum GWA and 
assume that $h = \sum h_i t^i$ is not a unit.
Let $\ell = \gcd\{i-j \mid h_i h_j \neq 0\}$. 
If $h$ is a monomial, set $C_\ell=\kk^\times$ and otherwise let $C_\ell$ be the subgroup of $\kk^\times$ consisting of $\ell^\text{th}$ roots of unity. If $(\gamma,\mu) \in C_\ell \times \kk^\times$ and $i_0 \in \supp(h)$, then there is an automorphism $\eta_{\gamma,\mu}$ of $A$ such that 
\begin{align}
\label{eq.eta}
\eta_{\gamma,\mu}(t) = \gamma t,\quad 
\eta_{\gamma,\mu}(v) = \mu v, \quad
\eta(u) = \mu^{-1}\gamma^{i_0} u.
\end{align}
The definition of $\eta_{\gamma,\mu}$ is independent of the choice of $i_0$ (see \cite[Remark 2.4]{GHo}), hence we typically take $i_0=\deg_t(h)$.


Let $G=\{ \eta_{\gamma,\mu} \mid (\gamma, \mu) \in C_{\ell} \times \kk^\times\}$, which is a subgroup of $\Aut(A)$. When $q \neq -1$, then $\Aut(A)=G$ \cite[Theorem B]{SAV}. If $q=-1$, then there is an order 2 automorphism $\Omega$ defined by
\[ \Omega(t)=-t, \quad \Omega(v)=u, \quad \Omega(u)=v.\]
Then there is a short exact sequence,
\[
0 \longrightarrow G \longrightarrow \Aut(A) \longrightarrow \ZZ/\ZZ_2 \longrightarrow 0.
\]
In this case, every automorphism of $A$ is either some $\eta_{\gamma,\mu}$ or else $\Omega \circ \eta_{\gamma,\mu}$ \cite[Proposition 2.7]{GHo}. 
We note that this classification shows that every automorphism of a quantum GWA is weakly $\ZZ$-graded (and when $q \neq -1$, every automorphism is actually $\ZZ$-graded).
We also remark that work in \cite{CGWZ} classifies the automorphism groups for certain higher rank quantum GWAs.

\begin{theorem}
\label{thm.qthick}
Let $A = \kk[t](u,v,\sigma,h)$ be a quantum GWA with $\sigma(t)=qt$ where $q^2 \neq 1$.
Let $G=\grp{\eta_{\gamma,\mu}}$ be a cyclic subgroup of $\Aut(A)$ of order $n$. Let $m = \ord(\gamma)$ so $m \mid n$.

(A) The action of $G$ is the restriction of the action to the group of group-likes of an inner-faithful weakly $\ZZ$-graded $T_n(\lambda, m)$-module algebra action if and only if there exists an integer $k$ coprime to $m$ such that $\mu q^k$ is an $m^\text{th}$ root of unity.

(B) The actions of each $T_n(\lambda, m)$ whose group-like elements restrict to the action of $G$
are parameterized by nonzero polynomials $\phi(t) \in \kk[t]$ of degree $d$ such that
\begin{enumerate}
    \item $\gcd(d-1,m)=1$, and
    \item $\supp(\phi)$ is contained in a single congruence class modulo $m$.
\end{enumerate}
\end{theorem}

Since the GWAs $A$ we consider have base ring $A_0 = \kk[t]$ (the commutative polynomial ring in one variable), in Section \ref{sec.kt} we consider generalized Taft actions on $\kk[t]$. In this setting, without any assumptions on linearity, we provide a full classification of inner-faithful actions up to conjugation by an automorphism of $\kk[t]$ (Proposition \ref{prop.ktact}).

In Section \ref{sec.GWA} we prove our main results above, Theorems~\ref{thm.intro} and \ref{thm.qthick}. These are largely a consequence of Theorem~\ref{thm.main} which, along with Theorem \ref{thm.omega}, fully classify inner-faithful actions of generalized Taft algebras on quantum GWAs at a root of unity $q \neq 1$ that weakly preserve the $\ZZ$-grading.


Finally, in Section \ref{sec.fixed} we compute the invariants for Taft actions on quantum GWAs under Taft algebra actions and show that generically the invariant rings are commutative rings whose associated graded rings are Kleinian singularities (Theorem \ref{thm.fixed}). 

\subsection*{Acknowledgments} 
R. Won was partially supported by an AMS--Simons Travel Grant. The authors thank the referee for several suggestions and corrections.

\section{Generalized Taft actions on $\kk[t]$}
\label{sec.kt}

Throughout, all algebras are associative $\kk$-algebras over a base field $\kk$ of characteristic zero. We begin this section with some background on Hopf actions and generalized Taft algebras. Subsequently, we classify actions of generalized Taft algebras on the polynomial ring $\kk[t]$.

A Hopf algebra is a bialgebra with antipode. 
Given a Hopf algebra $H$, we denote the comultiplication, counit, and antipode by $\Delta$, $\epsilon$, and $S$, respectively. An element $g \in H$ is \emph{grouplike} if $\Delta(g)=g \tensor g$. Given grouplikes $g,h \in H$, an element $x \in H$ is \emph{$(g,h)$-skew-primitive} provided $\Delta(x)=g \tensor x + x \tensor h$.
For any $(g,h)$-skew-primitive element $x$, $\epsilon(x)=0$ and $S(x)=-g\inv x h\inv$.

An algebra $A$ is a \emph{left $H$-module algebra} if $A$ is a left $H$-module with action $h \tensor a \mapsto h(a)$ such that $h(1_A) = \epsilon(h)1_A$ and $h(ab) = \sum h_1(a)h_2(b)$
for all $a,b \in A$ and $h \in H$.
Equivalently, $A$ is an algebra object in the monoidal category of left $H$-modules.

We say that a Hopf algebra $H$ is \emph{pointed} if every simple $H$-comodule is one-dimensional. A partial but extensive classification of finite-dimensional pointed Hopf algebras has been obtained by Andruskiewitsch and Schneider \cite{AS1}. 
Amongst the most fundamental examples of pointed Hopf algebras are the (generalized) Taft algebras presented in \eqref{eq.taft}. Indeed, $T_n(\lambda,m)$ is a Hopf algebra in which $g$ is grouplike and $x$ is $(g,1)$-skew primitive. That is, $\Delta(g)=g \tensor g$ and $\Delta(x)=g \tensor x + x \tensor 1$. In fact, the element $g$ generates the group of grouplikes in $T_n(\lambda,m)$.
When $m=n$, we say that $T_n(\lambda)=T_n(\lambda,n)$ is a \emph{Taft algebra}.
There is an even more general definition of a generalized Taft algebra, but we only consider those defined in \eqref{eq.taft}.

Let $H$ be a Hopf algebra and let $M$ be an $H$-module. We say $M$ is an \emph{inner-faithful} $H$-module (or that the action is inner-faithful) if $IM \neq 0$ for every Hopf ideal $I$ of $H$. If $A$ is an $H$-module algebra, then we say that $A$ is inner-faithful if it is inner-faithful as an $H$-module.
By \cite[Corollary 3.7]{cline}, a $T=T_n(\lambda,m)$-module $M$ is inner-faithful if and only if the group of grouplikes $\grp{g}$ acts faithfully and $x(M)\neq 0$. If $m=n$, so $T$ is a Taft algebra, then $x(M) = 0$ whenever $\grp{g}$ does not act faithfully (see \cite[Lemma 2.5]{KW},\cite[Corollary 3.2]{BM}). Thus, in this case, the faithfulness condition on $\grp{g}$ is superfluous.

Before we can study generalized Taft actions on quantum GWAs, we first consider actions on the polynomial ring $\kk[t]$. Throughout this section, let $m$ and $n$ be positive integers such that $m > 1$ and $m \mid n$. Let $\lambda$ be a primitive $m^\text{th}$ root of unity.
Set $T=T_n(\lambda,m)$ as in \eqref{eq.taft}.
Fix $\gamma \in \kk^{\times}$. For $k \in \NN$, define the \emph{$k^\text{th}$ $\gamma$-number} to be
\[ [k]_\gamma = \sum_{j=0}^{k-1} \gamma^j.\]
When $\gamma \neq 1$, then $[k]_\gamma = (1-\gamma^k)/(1-\gamma)$.
Observe that if $\gamma \neq 1$ is a primitive $\og^\text{th}$ root of unity, and $k \equiv m \pmod{\og}$, then $[k]_{\gamma} = [m]_{\gamma}$.
For $f \in \kk[t]$, let $\qder_\gamma(f)$ denote the \emph{$\gamma$-derivative} of $f$, defined by
\[
\delta_\gamma(f) = \frac{f(t)-f(\gamma t)}{t-\gamma t} \quad \text{when $\gamma \neq 1$}
\]
and define $\delta_1(f)=f'$. Note that $\delta_{\gamma}(t^k) = [k]_{\gamma} t^{k-1}$ for all values of $\gamma$.
The $\gamma$-derivative satisfies a $\gamma$-analog of the Leibniz rule, namely
\begin{align*}\delta_{\gamma}(f(t) g(t)) = \qder_{\gamma}(f(t))g(t) + f(\gamma t) \qder_{\gamma}(g(t)) = \qder_{\gamma}(f(t))g(\gamma t) + f( t) \qder_{\gamma}(g(t)).
\end{align*}
We remark that $\qder_{\gamma}$ is nilpotent of index $\ord(\gamma)$. If $k < \ord(\gamma)$, then $\qder_{\gamma}^{\ord(\gamma)}(t^k) = 0$ and if $k \geq \ord(\gamma)$ then because $[\ord(\gamma)]_{\gamma} = 0$ we have
\[ \qder_\gamma^{\ord(\gamma)} (t^k) = \left(\prod_{i=1}^{\ord(\gamma)}[i]_{\gamma}\right) t^{k-\ord(\gamma)} = 0.\]

In order to classify $T$ actions on $\kk[t]$, we first need a technical lemma. Let $f=\sum_{i=0}^d f_i t^i \in \kk[t]$ be a polynomial of degree $d$. We will abuse derivative notation in what follows by writing
\[ f^{(k)} = \sum_{i=0}^{d-k} f_{i+k} t^i.\]

\begin{lemma}
\label{lem.xtd}
Suppose $\kk[t]$ is a left $T$-module algebra, $g(t)=\gamma t$ for some $\gamma \in \kk^\times$, and $x(t)=\phi \in \kk[t]$.
\begin{enumerate}
\item \label{xtd1} If $\gamma=1$, then $\phi = 0$. 
\item \label{xtd2} For all $f \in \kk[t]$ we have
\begin{align}\label{eq.xtn}
x(f)=\phi \delta_\gamma(f).
\end{align}
\item \label{xtd3} If $\gamma \neq 1$ is a primitive $\og^\text{th}$ root of unity and $\deg(f)=k$ with $\supp(f)$ contained in a single congruence class mod $\og$, then
\[ x(f) = [k]_{\gamma} \phi f^{(1)}.\]
\end{enumerate}
\end{lemma}
\begin{proof}
\eqref{xtd1}
Assume $\gamma = 1$. Since $\kk[t]$ is a $T$-module, then
\[
0 = (gx-\lambda xg)(t) = (1-\lambda)\phi,
\]
and since $\lambda \neq 1$, we have $\phi = 0$.

\eqref{xtd2} When $\gamma=1$, by part \eqref{xtd1}, $\phi = 0$ and so equation \eqref{eq.xtn} follows. So assume $\gamma \neq 1$. Note that \eqref{eq.xtn} holds for $f=t$ by hypothesis. Suppose it holds for $t^k$ with $k \geq 1$. Then we have
\[ x(t^{k+1}) 
    = g(t) \cdot x(t^k) + x(t) \cdot t^k 
    = \gamma tx(t^k) + \phi t^k
    = \gamma t \phi\delta_\gamma(t^k) + \phi t^k
    = \phi \delta_\gamma(t^{k+1}).
\] 
Thus, \eqref{eq.xtn} follows by induction and linearity.
Now \eqref{xtd3} follows from \eqref{xtd2}.
\end{proof}

In Lemma \ref{lem.xtd} \eqref{xtd3}, we assumed that $\deg(f)=k$ and $\supp(f)$ is contained in a single congruence class mod $\og$. This implies that $f^{(1)}$ is a polynomial of degree $k-1$ all of whose nonzero terms occur in degrees which are congruent to $k-1$ modulo $\og$.


Suppose $\kk[t]$ is a $T$-module algebra.
Since $g \in T$ is grouplike then it acts as an automorphism on $\kk[t]$. That is, $g(t)=\gamma t + \kappa$ for some $\gamma \in \kk^\times$ and $\kappa \in \kk$. Since $g$ has finite order, then necessarily either $g$ is the identity or else $\gamma \neq 1$.
If $\gamma \neq 1$, then up to conjugation by an algebra automorphism of $\kk[t]$, we may assume that $g$ acts on $t$ by multiplication by a scalar.
However, we do not make any assumption about the action of $x$. Hence, our next result classifies inner-faithful generalized Taft actions which make $\kk[t]$ a $T$-module algebra, up to conjugation by an automorphism of $\kk[t]$.

\begin{proposition}\label{prop.ktact}
Let $T=T_n(\lambda,m)$ be a generalized Taft algebra. Let $\gamma \in \kk \setminus\{0,1\}$ and $\phi \in \kk[t]$ 
be a nonzero polynomial of degree $d$. 
If $\kk[t]$ is a $T$-module algebra with $g(t)=\gamma t$ and $x(t)=\phi$, then
\begin{enumerate}
    \item \label{prop.ktact.cond1} $\gamma$ is a primitive $m^\text{th}$ root of unity,
    \item \label{prop.ktact.cond2} $\lambda = \gamma^{d-1}$ and $\gcd(d-1, m) = 1$, and
    \item \label{prop.ktact.cond3} $\supp(\phi) \subseteq \{d, d - m, d - 2m, \hdots \}$.
\end{enumerate}
Furthermore, the action is inner-faithful if and only if $m=n$.


Conversely, if $\gamma$ and $\phi$ satisfy the conditions (1)–-(3), then there is a unique $T$-module algebra structure on $\kk[t]$ such that $g(t) = \gamma t$ and $x(t) = \phi$.
\end{proposition}
\begin{proof}
Suppose first that $\kk[t]$ is a $T$-module algebra with $g(t) = \gamma t$ and $x(t) = \phi$. Since $g^n = 1$, we have $g^n(t) = \gamma^n t = t$ and hence $\gamma$ is a root of unity of order $\og \mid n$. Furthermore we have
\begin{align*}
0 = (gx-\lambda xg)(t)     
    =  g(\phi(t))-\lambda x(\gamma t)
    = \phi(\gamma t) - \lambda \gamma \phi(t).
\end{align*}
Writing $\phi=\sum_{i=0}^d \phi_i t^i$,  we must have $\sum_{i=0}^d (\gamma^i - \lambda \gamma) \phi_i t^i =  0$. Therefore,
\begin{align}
\label{eq.lamgam}
\lambda = \gamma^{i-1} \quad \text{ for all } \quad i \in \supp(\phi).
\end{align}
In particular, $\lambda = \gamma^{d-1}$ and all elements of $\supp(\phi)$ are congruent to $d$ modulo $\og$. Hence, $\ord(\lambda)=m$ must divide $\og$ and so $\gcd(d-1,\og)=\og/m$.

If $d \equiv 0 \pmod \og$, then by \eqref{eq.lamgam}, $\lambda=\gamma\inv$ and so $\gamma$ is a root of unity of order $m$. Otherwise, $d>0$.
By Lemma \ref{lem.xtd} and induction, we compute that
\begin{align}
\label{eq.xmt}
x^m(t) = \left(\prod_{i=0}^{m-1} [1+i(d-1)]_\gamma\right) f
\end{align}
for some $f \in \kk[t]$. Since $d>1$ by hypothesis, then a degree argument shows that $f \neq 0$. Since $x^m(t)=0$, then there exists some $0 \leq i \leq m -1$ such that $[1+i(d-1)]_\gamma = 0$. This happens if and only if $1 = \gamma^{1+i(d-1)} = \gamma \lambda^i$.
Therefore, $\og \mid m$ and since $m \mid \og$, we conclude that $m=s$. The action is inner-faithful if and only if $g$ acts faithfully and $x(A) \neq 0$. This happens if and only if $\og = m = n$.

Conversely, suppose that conditions \eqref{prop.ktact.cond1}--\eqref{prop.ktact.cond3} hold.
We will show that $\kk[t]$ admits the structure of a $T$-module algebra where $g(t) = \gamma t$ and $x(t) = \phi$. It suffices to show that $g^n(t) = t$, $(gx - \lambda xg)(t) = 0$, and $x^m(t) = 0$. If so then $\kk[t]_1$ is a $T$-module, and there is a unique $T$-module algebra structure on $\kk[t]$ given by extending the actions of $g$ and $x$ in the unique way that is compatible with the coproduct of $T$. For example, we have $g(t^2) = g(t)g(t) = \gamma^2 t$ and $x(t^2) = g(t)x(t) + x(t)t = (1 + \gamma)t\phi$.

Since $\gamma$ is a primitive $m^\text{th}$ root of unity and $m \mid n$, it is clear that $g^n(t) = \gamma^n t = t$. 
Further,
\[(gx-\lambda xg)(t)     
    =  \phi(\gamma t) - \lambda \gamma \phi(t)
\]
and this is $0$ by conditions \eqref{prop.ktact.cond2} and \eqref{prop.ktact.cond3}.

Finally, if $d \equiv 0 \pmod{m}$,  then $x(t) = \phi$ has degree $d$ and so by the same computation as in Lemma~\ref{lem.xtd},
$x^2(t) = 0$. Since $m > 1$, we have $x^m(t) = 0$.
Otherwise, $d > 0$ and so by the same computation as in Lemma~\ref{lem.xtd}, we have \eqref{eq.xmt}
for some $f \in \kk[t]$. Since $\gcd(d-1,m) = 1$, one of the $m$ factors in the product is equal to $0$, and hence $x^m(t) = 0$.
\end{proof}

\begin{remark}
If $\phi = 0$, then as long as $\gamma$ is an $n$th root of unity, then $\kk[t]$ is a $T_n(\lambda, m)$-module algebra with $g(t) = \gamma t$ and $x(t) = 0$. However, this action is not inner-faithful.
\end{remark}

\section{Generalized Taft actions on quantum generalized Weyl algebras}
\label{sec.GWA}


Fix a quantum GWA $A = \kk[t](u,v,\sigma,h)$ where $\sigma(t) = qt$ for some root of unity $q \neq 1$ and a generalized Taft algebra $T = T_n(\lambda, m)$.
We restrict our attention to actions of $T$ on $A$ that are \emph{weakly} $\ZZ$-graded (as described in the introduction). Hence, these actions extend the actions of $T$ on $\kk[t]$ studied above. This is also a natural restriction because it generalizes group actions that are $\ZZ$-graded up to a group automorphism of $\ZZ$. By the classification of Su\'{a}rez-Alvarez and Vivas \cite[Proposition 2.7]{SAV}, every automorphism of $A$ is in fact weakly $\ZZ$-graded.

Since $g \in T$ is grouplike, then it acts on $A$ as an automorphism of the form $\eta_{\gamma,\mu}$ as in equation~\eqref{eq.eta} or, in the case $q=-1$, as either $\eta_{\gamma, \mu}$ or $\Omega \circ \eta_{\gamma,\mu}$. 
To be weakly $\ZZ$-graded, considering the action of $x$ on $A$, we must have $x(t) \in A_0$ and $x(u), x(v) \in A_{-1} \oplus A_{1}$. We summarize our standing hypotheses as follows.

\begin{hypothesis}
\label{hyp.taft}
Let $h = \sum_{i=0}^D h_i t^i \in \kk[t]$ be a polynomial of degree $D>0$, let $q \in \kk^\times$ be a root of unity of order $\ord(q)>1$, and let $\sigma \in \Aut(\kk[t])$ be defined by $\sigma(t) = qt$. We consider the quantum GWA $A = \kk[t](u,v, \sigma, h)$ as in \eqref{eq.qgwa}.
Let $m$ and $n$ be positive integers such that $m > 1$ and $m \mid n$, and let $\lambda$ be a primitive $m^\text{th}$ root of unity. Let $T=T_n(\lambda,m)$ be a generalized Taft algebra as in \eqref{eq.taft} and assume that $A$ is a weakly $\ZZ$-graded $T$-module algebra in the sense that $T(A_0) \subset A_0$ and $T(A_{-i} \oplus A_{i}) \subset A_{-i} \oplus A_{i}$ for all $i \in \mathbb{N}$. That is:
\begin{itemize}
\item for some $\mu \in \kk^\times$ and $\gamma$ an $\og^\text{th}$ root of unity, $g$ acts as an automorphism of the form $\eta_{\gamma,\mu}$ as in equation~\eqref{eq.eta} or, in the case $q=-1$, as either $\eta_{\gamma,\mu}$ or $\Omega \circ \eta_{\gamma,\mu}$, and
\item $x$ acts on $A$ via
\begin{align}
\label{eq.xact}
x(t) = \phi(t) = \sum_{i=0}^d \phi_i t^i, \quad 
\quad x(u) = \alpha_{11}u + \alpha_{21}v, \quad
x(v) = \alpha_{12}u + \alpha_{22}v,
\end{align}
where for each $0 \leq i \leq d$, $\phi_i \in \kk$ with $\phi_d \neq 0$ and $\alpha_{ij} \in \kk[t]$ for $i,j \in \{1,2\}$.
\end{itemize}
\end{hypothesis}

For most of this section we make no restriction on $q$ but assume that $g$ acts as an automorphism $\eta_{\gamma, \mu}$ as in equation \eqref{eq.eta}. 
If $q \neq -1$, then this is automatic. For the case $q = -1$, we consider the case where $g$ acts as $\Omega \circ \eta_{\gamma,\mu}$ in Theorem \ref{thm.omega}.

\begin{lemma}\label{lem.special}
Assume Hypothesis \ref{hyp.taft} with $g$ acting as an automorphism of the form $\eta_{\gamma,\mu}$. Then     $\alpha_{12}=\alpha_{21}=0$.
In particular, $A$ is a $\ZZ$-graded $T$-module.
Moreover, if $\gamma=1$, then $x(t)=0$ and the action of $T$ on $A$ is not inner-faithful.
\end{lemma}
\begin{proof}
Since $A$ is a $T$-module algebra, then the action of $x$ necessarily preserves the relations of $A$. In particular,
\begin{align*}
0 &= x(vu-h) 
    = (\mu\sigma\inv(\alpha_{11})+\alpha_{22})h + \mu\sigma\inv(\alpha_{21}) v^2 + \alpha_{12} u^2 - x(h) \\
\begin{split}
0 &= x(uv-\sigma(h)) 
    = (\mu\inv \gamma^{D}\sigma(\alpha_{22})+\alpha_{11})\sigma(h) + \alpha_{21} v^2 \\
    &\hspace{15em} + \mu\inv \gamma^{D} \sigma(\alpha_{12}) u^2 - x(\sigma(h)).
\end{split}
\end{align*}
That $\alpha_{12}=\alpha_{21}=0$ now follows from the $\ZZ$-grading on $A$. Observe that this shows that although \emph{a priori}, the action of $T$ on $A$ is only weakly $\ZZ$-graded, it is in fact $\ZZ$-graded in the usual sense.

Suppose $\gamma=1$. Then $x(t)=0$ by Lemma~\ref{lem.xtd}. Now $x(u)=\alpha_{11} u$ and so $x^m(u)=\alpha_{11}^m u$. Thus, $x^m(u)=0$ if and only if $\alpha_{11}=0$. Similarly, $x^m(v)=0$ if and only if $\alpha_{22}=0$. In this situation, the action of $T$ is not inner-faithful.
\end{proof}

Suppose that $T$ and $A$ satisfy Hypothesis~\ref{hyp.taft}. If $g$ acts as an automorphism $\eta_{\gamma, \mu}$, then in light of Lemma~\ref{lem.special}, the following equations are satisfied:
\begin{align}
\label{eq.vuh}
0 &= x(vu-h) = (\mu\sigma\inv(\alpha_{11})+\alpha_{22})h - x(h) \\
\label{eq.uvh}
0 &= x(uv-\sigma(h)) 
    = (\mu\inv \gamma^{D}\sigma(\alpha_{22})+\alpha_{11})\sigma(h) - x(\sigma(h)) \\
\label{eq.gxu}
0 &= (gx-\lambda xg)(u) 
    = g(\alpha_{11} u) - \lambda x(\mu\inv \gamma^{D} u)
    = \mu\inv \gamma^{D} (g(\alpha_{11}) - \lambda \alpha_{11}) u \\
\label{eq.gxv}
0 &= (gx-\lambda xg)(v)
    = g(\alpha_{22} v) - \lambda x(\mu v)
    = \mu (g(\alpha_{22}) - \lambda \alpha_{22}) v \\
\label{eq.gxt}
0 &= (gx-\lambda xg)(t)     
    =  g(\phi(t))-\lambda x(\gamma t)
    = \phi(\gamma t) - \lambda \gamma \phi(t) \\
\label{eq.ut}
0 &= x(ut-qtu) 
    = (\mu\inv \gamma^{D}\phi(qt) - q\phi(t))u + \alpha_{11}q(1-\gamma)tu
    \\
\label{eq.vt} 
0 &= x(vt-q\inv tv) 
    = (\mu\phi(q\inv t)-q\inv\phi(t))v + \alpha_{22}q\inv(1-\gamma)tv \\
\label{eq.xnilp}
0 &= x^n(t) = x^n(u) = x^n(v).
\end{align}

Henceforth, we will assume $\gamma \neq 1$.

\begin{lemma}\label{lem.equiv}
Assume Hypothesis \ref{hyp.taft} with $g$ acting as $\eta_{\gamma,\mu}$ and with $\gamma \neq 1$.
Then 
\begin{enumerate}
    \item $\mu$ is a root of unity and $\ord(\mu) \mid n$,
    \item $m=s$,
    \item $\supp(\phi)$ is contained in a single congruence class modulo $m$,
    \item $\lambda=\gamma^{d-1}$ and so $\gcd(d-1, m) = 1$, and
    \item $\supp(h)$ is contained in a single congruence class modulo $m$.
\end{enumerate}
\end{lemma}
\begin{proof}
The automorphism $\eta_{\gamma,\mu}$ has order dividing $n$ if and only if $\ord(\gamma) \mid n$ and $\ord(\mu) \mid n$, which gives statement (1).
Under Hypothesis~\ref{hyp.taft}, $A_0 = \kk[t]$ is a $T$-module algebra with $g(t) = \gamma t$ and $x(t) = \phi$. Hence, Proposition \ref{prop.ktact} yields statements (2)--(4).

If $h$ is a monomial, then (5) is trivial. Otherwise, recall that $\ell = \gcd\{ i - j \mid h_i h_j \neq 0\}$, and for $\eta_{\gamma, \mu}$ to be an automorphism of $A$, we must have $\gamma^{\ell} = 1$.
Therefore $\ell$ must be a multiple of $m$ and the support of $h$ is also contained in a single congruence class modulo $m$.
\end{proof}

In the remainder of this section, we again abuse derivative notation and write
\[
\phi^{(k)} = \sum_{i=0}^{d-k} \phi_{i+k} t^i
\quad\text{and}\quad
h^{(k)} = \sum_{i=0}^{D-k} h_{i+k} t^i.
\]
Assuming Hypothesis \ref{hyp.taft}, then by Lemma \ref{lem.equiv} the support of $h$ and $\phi$ are each contained in a single congruence class modulo $\og$. Hence, by Lemma \ref{lem.xtd} \eqref{xtd3}, we have
\[ x(h) = [D]_{\gamma} \phi h^{(1)} \quad\text{and}\quad x(\phi) = [d]_{\gamma} \phi \phi^{(1)}.\]

Next, we determine the parameters $\alpha_{11}$ and $\alpha_{22}$ from equation~\eqref{eq.xact}.

\begin{lemma}
\label{lem.alpha}
Assume Hypothesis \ref{hyp.taft} with $g$ acting as $\eta_{\gamma,\mu}$ and with $\gamma \neq 1$. 
\begin{enumerate}
\item We have
\begin{align}
\label{eq.a22}
\alpha_{22} &= \frac{1-\mu q^{1-d}}{1-\gamma}\phi^{(1)} \\
\label{eq.a11}
\alpha_{11} &= \frac{1-\mu\inv \gamma^{D}q^{d-1}}{1-\gamma} \phi^{(1)}.
\end{align}
Hence, for $i=1,2$, $\alpha_{ii}$ is supported on the congruence class of $d-1$ modulo $m$.

\item If $\alpha_{11} \neq 0$ or $\alpha_{22} \neq 0$, then $t \mid \phi$.

\item If the support of $\alpha_{11}$ is $0$ modulo $\om$, then $x^m(u) = 0$ if and only if $\alpha_{11} = 0$. If the support of $\alpha_{11}$ is not $0$ modulo $\om$, then $x^m(u) = 0$ if and only if $\mu q^{1-d}$ is an $m^\text{th}$ root of unity (not necessarily primitive).

\item If the support of $\alpha_{22}$ is $0$ modulo $\om$, then $x^n(v) = 0$ if and only if $\alpha_{22} = 0$. If the support of $\alpha_{22}$ is not $0$ modulo $\om$, then $x^n(v) = 0$ if and only if $\mu q^{1-d}$ is an $m^\text{th}$ root of unity (not necessarily primitive).
\end{enumerate}
\end{lemma}
\begin{proof}
We prove the statements for $v$ and $\alpha_{22}$. The statements for $u$ and $\alpha_{11}$ follow similarly.

As mentioned after Lemma~\ref{lem.special}, under these hypotheses, the $T$-action on $A$ satisfies equations \eqref{eq.vuh}--\eqref{eq.xnilp}.
By \eqref{eq.vt} and Lemma \ref{lem.equiv},
\begin{align*}
\alpha_{22}q\inv(1-\gamma)t
    &= q\inv\phi(t)-\mu\phi(q\inv t) \\
    &= q\inv \sum_{i=0}^d \phi_i t^i - \mu \sum_{i=0}^d \phi_i (q\inv t)^i \\ 
    &= q\inv \sum_{i=0}^d (1-\mu q^{1-i})\phi_i t^i.
\end{align*}
So, either $\alpha_{22}=0$, in which case $\mu=q^{i-1}$ for all $i \in \supp(\phi)$, or else $t \mid \phi$ and
\[ \alpha_{22} = \sum_{i=1}^d \frac{(1-\mu q^{1-i})}{1-\gamma}\phi_i t^{i-1}.\]

This proves (2). Since the support of $\phi$ is contained in a single congruence class modulo $\ord(\gamma)$, then it follows from the above computation that the support of $\alpha_{22}$ has the same property. Note $\ord(\gamma)=m$ by Lemma \ref{lem.equiv} (2).

First, suppose that the support of $\alpha_{22}$ is $0$ modulo $\om$. Then $\qder_{\gamma}(\alpha_{22}) = 0$. By an induction argument, for each $k \geq 1$, $x^k(v) = \alpha_{22}^k v$ and so the first part of (4) follows.

Now suppose that the support of $\alpha_{22}$ is not $0$ modulo $\om$. Set $\beta_0=1$ and for $k > 0$, let
\[
\beta_k = \sum_{i=1}^d \frac{(1-\gamma^{(k-1)(d-1)}\mu q^{1-i})}{1-\gamma}\phi_i t^{i-1},
\]
so that $\beta_1=\alpha_{22}$. 
We claim that
\begin{align}
\label{eq.beta}
x(\beta_k \cdots \beta_1 \beta_0 v)
    = \beta_{k+1} \cdots \beta_1 \beta_0 v.
\end{align}
Clearly this is true for $k=0$. Then
\begin{align*}
x(\beta_{k+1} \cdots \beta_1 v)
    &= g(\beta_{k+1})x(\beta_k \cdots \beta_1 v)
        + x(\beta_{k+1})\beta_k \cdots \beta_1 v \\
    &= (\gamma^{d-1}\beta_{k+1})(\beta_{k+1} \cdots \beta_1 v) + ([d-1]_\gamma \phi \beta_{k+1}^{(1)})\beta_k \cdots \beta_1 v \\
    &= (\gamma^{d-1}\beta_{k+1} + [d-1]_\gamma \phi^{(1)})\beta_{k+1} \cdots \beta_1 v \\
    &= \left( \sum_{i=1}^d \frac{\gamma^{d-1}(1-\gamma^{ k(d-1)} \mu q^{1-i})+(1-\gamma^{d-1}) }{1-\gamma} \phi_i t^{i-1} \right)\beta_{k+1} \cdots \beta_1 v \\
    &= \beta_{k+2} \cdots \beta_1 v.
\end{align*}
Now \eqref{eq.beta} implies that $x^m(v)=\beta_m \cdots \beta_1 v$. 

Hence, $x^m(v)=0$ if and only if $\beta_k=0$ for some $k=1,\dots,m$. Moreover, $\beta_k=0$ if and only if
\[
\gamma^{(k-1)(d-1)} = \mu\inv q^{i-1}
\]
for all $i \in \supp(\phi)$.
This implies that $\mu\inv q^{i-1}=\mu\inv q^{d-1}$ for all $i \in \supp(\phi)$, which reduces to $q^d=q^i$. Hence we obtain \eqref{eq.a22}, completing our proof of (1).

Finally, recall from Lemma~\ref{lem.equiv} (4) that $\gamma^{d-1}$ is a primitive $m^\text{th}$ root of unity and so the $\gamma^{(k-1)(d-1)}$ range over the $m^\text{th}$ roots of unity (other than $1$) exactly once. It follows that $x^m(v)=0$ if and only if $\mu q^{d-1}$ is an $m^\text{th}$ root of unity. Since $\ord(\gamma)=m$, then this reduces to $x^m(v)=0$ if and only if $\mu q^{1-d}$ is an $m^\text{th}$ root of unity which completes our proof of (4).
\end{proof}

The following theorem essentially summarizes the above work on necessary conditions for an action. The interesting part is that these conditions are in fact sufficient to define an $T$-module algebra structure on a quantum GWA.

\begin{theorem}\label{thm.main}
Fix a GWA $A = \kk[t](u,v,\sigma,h)$ with defining polynomial $h$ of degree $D>0$, and defining automorphism $\sigma(t) = qt$ where $q \in \kk^\times$ is a root of unity, $q \neq 1$. 
Let $T=T_n(\lambda,m)$ be a generalized Taft algebra.

Suppose that $A$ is an inner-faithful weakly $\ZZ$-graded $T$-module algebra where $g$ acts as an automorphism $\eta_{\gamma,\mu}$ as in equation \eqref{eq.eta} with $\gamma \neq 1$, and $x(t) = \phi(t) \in \kk[t]$. Then
\begin{enumerate}
\item \label{cond.0} the automorphism $\eta_{\gamma,\mu}$ has order $n$,
\item \label{cond.1} $\supp(h)$ is contained in a single congruence class modulo $m$,
\item \label{cond.3} $\phi(t)$ is a nonzero polynomial of degree $d$ whose support is contained in a single congruence class modulo $m$,
\item \label{cond.2} $\ord(\gamma) = m$ and $\lambda = \gamma^{d-1}$,
\item \label{cond.4} $\mu q^{1-d}$ is an $m^\text{th}$ root of unity,
\item \label{cond.5} $\displaystyle x(u) = \frac{1 - \mu \inv \gamma^D q^{d-1}}{1-\gamma} \phi^{(1)} u =: \alpha_{11} u$, and
\item \label{cond.6} $\displaystyle x(v) = \frac{1 - \mu q^{1-d}}{1-\gamma} \phi^{(1)} v =: \alpha_{22} v$.
\end{enumerate}
In particular, $A$ is actually a $\ZZ$-graded $T$-module algebra in the usual sense.

Conversely, suppose that $\eta_{\gamma,\mu} \in \Aut(A)$ with $\gamma \neq 1$, $\phi(t) \in \kk[t]$, and the parameters $h$ and $q$ satisfy conditions \eqref{cond.0}--\eqref{cond.4}. Let $g$ act on $A$ via $\eta_{\gamma,\mu}$. Define an action of $x$ on $A$ by setting $x(t) = \phi(t)$, $x(u)$ to be as in \eqref{cond.5}, and $x(v)$ to be as in \eqref{cond.6}. 
Then it is possible to extend this action to all of $A$ to make $A$ an inner-faithful $\ZZ$-graded $T$-module algebra.
\end{theorem}

\begin{proof}
Suppose $A$ is an inner-faithful weakly $\ZZ$-graded $T$-module where $g$ acts as an automorphism $\eta_{\gamma,\mu}$ with $\gamma \neq 1$.
Since the $T$-action is weakly $\ZZ$-graded, we must have $x(t) = \phi(t)$ for some $\phi(t) \in \kk[t]$. Then conditions \eqref{cond.0}--\eqref{cond.6} follow from Lemmas~\ref{lem.equiv} and \ref{lem.alpha}.

Conversely, suppose that $\eta_{\gamma,\mu} \in \Aut(A)$ with $\gamma \neq 1$, $\phi(t) \in \kk[t]$, and the parameters $h$ and $q$ satisfy conditions \eqref{cond.0}--\eqref{cond.4}. We will show that $A$ is a weakly $\ZZ$-graded $T$-module algebra where $g$ acts via $\eta_{\gamma,\mu}$ and $x$ acts via $x(t) = \phi(t)$, $x(u)$ is as in \eqref{cond.5} and $x(v)$ is as in $\eqref{cond.6}$ (and the action of $x$ is extended to all of $A$ to make $A$ a $T$-module algebra).

Since we are assuming that $g$ acts via an automorphism of order $n$, it suffices to show that equations \eqref{eq.vuh}--\eqref{eq.xnilp} hold for the $T$-action on $A$. 

By Lemma \ref{lem.xtd} and conditions \eqref{cond.5} and \eqref{cond.6}, we have
\begin{align*}
(\mu\sigma\inv(\alpha_{11})+\alpha_{22})h
    &= \left( \mu \sigma\inv\left( 
        \frac{1-\mu\inv \gamma^{D}q^{d-1}}{1-\gamma} \phi^{(1)} \right)  
        + \left(\frac{1-\mu q^{1-d}}{1-\gamma}\right) \phi^{(1)} \right)h \\
    &= (1-\gamma)\inv \left( \mu q^{1-d} ( 1-\mu\inv \gamma^{D}q^{d-1}) + (1-\mu q^{1-d})\right) \phi^{(1)} h \\
    &= (1-\gamma)\inv \left( 1-\gamma^D\right) \phi^{(1)} h \\
    &= [D]_\gamma \phi^{(1)} h = x(h).
\end{align*}
Thus, \eqref{eq.vuh} is satisfied. One checks similarly that \eqref{eq.uvh} is satisfied.


By conditions \eqref{cond.2} and \eqref{cond.5}, $\alpha_{11} \in \kk[t]$ is a polynomial of degree $d-1$ whose support is contained in a single congruence class modulo $m$. Hence, by condition \eqref{cond.3}, we a have $g(\alpha_{11}) = \lambda \alpha_{11}$ and so equation \eqref{eq.gxu} is satisfied. A similar argument for $\alpha_{22}$ shows that equation \eqref{eq.gxv} is satisfied.

By conditions \eqref{cond.3} and \eqref{cond.2}, we have $\phi(\gamma t) = \gamma^{d} \phi(t) = \lambda \gamma \phi(t)$ and so equation \eqref{eq.gxt} is satisfied. The proof of Lemma~\ref{lem.alpha} shows that if conditions \eqref{cond.5} and \eqref{cond.6} hold, then equations \eqref{eq.ut} and \eqref{eq.vt} are satisfied. Finally, by condition \eqref{cond.3} and Lemma~\ref{lem.xtd}, we have $x^m(t) = 0$. By conditions \eqref{cond.5} and \eqref{cond.6} and Lemma~\ref{lem.alpha}, $x^m(u) = x^m(v) = 0$. Hence, equation \eqref{eq.xnilp} is satisfied.

By \cite[Corollary 3.7]{cline}, the $T$-action on $A$ is inner-faithful if and only if $g$ acts an automorphism of order $n$ and $x$ does not act identically as $0$ on $A$. Hence, by conditions \eqref{cond.0} and \eqref{cond.3}, $A$ is an inner-faithful $T$-module.
\end{proof}

\begin{remark}\label{rem.taft}
If $m=n$ in Theorem \ref{thm.main}, then $T_n(\lambda,m)$ is the Taft algebra $T_n(\lambda)$. Then $\ord(\mu) \mid n$ and so Theorem \ref{thm.main} (5) reduces to the condition that $q^{1-d}$ is an $n^\text{th}$ root of unity.
\end{remark}

We are now ready to prove our main theorems.

\begin{proof}[Proof of Theorem \ref{thm.intro}]
Part (B) follows directly from Theorem \ref{thm.main}. Hence, we prove part (A). 
It is clear from Theorem \ref{thm.main} that it is enough to have (1) and the following condition:

(2') there exists an integer $d$ and a root of unity $\mu$ such that $\gcd(d-1, m) = 1$, $\lcm(m, \ord(\mu)) = n$, and $(\mu q^{1-d})^m =1$. 

If (2) holds, then we can let $k = d-1$, and $\mu = q^{k} = q^{d-1}$ and obtain (2'). Conversely, if (2') holds, then setting $k = d - 1$, we have that $q^{km} = \mu^m$. But this means that there exists an $m$th root of unity $\zeta$ such that $q^k = \mu \zeta$. Now 
\[ \ord(\mu \zeta) = \lcm( \ord(\mu), \ord(\zeta)) = \lcm( \ord(\mu), m) = n.\qedhere\]
\end{proof}

\begin{proof}[Proof of Theorem \ref{thm.qthick}]
Again, part (B) follows directly from Theorem \ref{thm.main}. Hence, we prove part (A). 

In this case, we fix the parameters $h,q,\gamma,\mu$. Moreover, we have $m=\ord(\gamma)$ and $n=\ord(\eta_{\gamma,\mu})$. It is clear that the given condition is implied by Theorem \ref{thm.main}. 

Conversely, given the condition, then by the definition of $\eta_{\gamma,\mu}$, $\supp(h)$ is contained in a single congruence class modulo $m$. Further, we can let $k=d-1$ and choose a polynomial $\phi(t) \in \kk[t]$ such that $\phi(t)$ has degree $d$ and $\supp(\phi(t))$ is contained in a single congruence class modulo $m$. By choosing $\lambda=\gamma^k$ we have an action by $T_n(\lambda,m)$ on $A$.
\end{proof}

We next consider actions on higher rank quantum GWAs.

\begin{definition}
\label{defn.gengwa}
Let $q = (q_1,\hdots,q_\GWAdeg) \in (\kk\backslash\{0,1\})^\GWAdeg$, and let $h=(h_1(t),\hdots,h_\GWAdeg(t))$ be a $\GWAdeg$-tuple of non-constant polynomials in $\kk[t]$. Then \emph{quantum GWA of rank $\GWAdeg$} corresponding to this data is the $\kk$-algebra generated over $\kk[t]$ by $\bu=(u_1,\hdots,u_\GWAdeg)$ and $\bv=(v_1,\hdots,v_\GWAdeg)$ with relations
\begin{align*}
&u_it = q_i tx_i  &v_it = q_i\inv(r)y_i  &	&\text{ for all $i \in \{1,\hdots, \GWAdeg\}$} \\
&u_iv_i = h_i(q_i t)	&v_iu_i = h_i(t)	&	&\text{ for all $i \in \{1,\hdots, \GWAdeg\}$} \\
&[u_i,u_j]=[v_i,v_j]=[u_i,v_j]=0	& &	&\text{ for all distinct $i,j \in \{1,\hdots, \GWAdeg\}$}.
\end{align*}
We denote this algebra by $\kk[t](\bu,\bv,\sigma,h)$. 
\end{definition}

A rank $\GWAdeg$ GWA $A=\kk[t](\bu,\bv,\sigma,h)$ is naturally $\ZZ^\GWAdeg$-graded by setting $\deg(r)=\bzero$ for all $r \in R$, $\deg(u_i)=\be_i$, and $\deg(v_i)=-\be_i$. We denote by $A_i$ the (rank one) quantum GWA subalgebra of $A$ generated by $x_i,y_i,t$.

\begin{corollary}
\label{cor.rankp}
Let $A$ be a quantum GWA of rank $\GWAdeg$ with $q_i \neq -1$ for all $i$.
Let $T=T_n(\lambda,m)$ be a generalized Taft algebra.
Suppose each $A_i$ is an inner-faithful weakly $\ZZ$-graded $T$-module.
Then there is an inner-faithful 
$\ZZ^\GWAdeg$-graded action of $T$ on $A$ that restricts to the given actions on each of the subalgebras $A_i$.
\end{corollary}
\begin{proof}
By the hypotheses, $g$ acts as some $\eta_{\gamma_i,\mu_i}$ on $A_i$, with $\gamma_i,\mu_i \in \kk^\times$ satisfying the conditions of Theorem \ref{thm.main}. Thus, 
\[ (gx-\lambda xg)(t) = (gx-\lambda xg)(u_i)=(gx-\lambda xg)(v_i) = 0
\]
and $x^m(t)=x^m(u_i)=x^m(v_i)=0$ for each $i$. Moreover, $g(r)=x(r)=0$ for each defining relation $r$ of $A_i$. It remains only to check that $g$ and $x$ satisfy the commutation relations between the $A_i$.

Choose $i,j$ with $1 \leq i < j \leq \GWAdeg$. To simplify the exposition, we take $i=1$ and $j=2$. Set
\begin{align*}
    &g(u_1) = \mu_1\inv \gamma_1^{D_1} u_1, \quad
    g(v_1) = \mu_1 v_1, \quad
    x(u_1) = \alpha_{11}u_1, \quad
    x(v_1) = \alpha_{22}v_1, \\
    &g(u_2) = \mu_2\inv \gamma_2^{D_2} u_2, \quad
    g(v_2) = \mu_2 v_2, \quad
    x(u_2) = \beta_{11}u_2, \quad
    x(v_2) = \beta_{22}v_2,
\end{align*}
where $\alpha_{ii},\beta_{ii}$ satisfy the appropriate conditions of Theorem \ref{thm.main}. Then
\begin{align*}
x(v_1v_2-v_2v_1)
    &= \left((\mu_1 v_1)(\beta_{22}v_2) + (\alpha_{22}v_1)v_2\right) - \left( (\mu_2v_2)(\alpha_{22}v_1) + (\beta_{22}v_2)v_1
    \right) \\
    &= \left( (\mu_1 q^{1-d} - 1)\beta_{22} + (1-\mu_2q^{1-d})\alpha_{22}\right)v_1v_2 = 0, \\
x(u_1u_2-u_2u_1)
    &= \left((\mu_1\inv \gamma^{D_1} u_1)(\beta_{11}u_2) + (\alpha_{11}u_1)u_2\right) \\
    &\qquad - \left( (\mu_2\inv \gamma^{D_2} u_2)(\alpha_{11}u_1) + (\beta_{11}u_2)u_1
    \right) \\
    &= \left( (\mu_1\inv \gamma^{D_1} q^{d-1} - 1)\beta_{11} + (1-\mu_2\inv \gamma^{D_2} q^{d-1})\alpha_{11}\right)u_1u_2 = 0, \\
x(u_1v_2-v_2u_1)
    &= \left((\mu_1\inv \gamma_1^D u_1)(\beta_{22}v_2) + (\alpha_{11}u_1)v_2\right) - \left( (\mu_2v_2)(\alpha_{11}u_1) + (\beta_{22}v_2)u_1
    \right) \\
    &= \left( (\mu_1\inv q^{d-1}\gamma_1^D-1) \beta_{22} + (1-q^{1-d}\mu_2)\alpha_{11}
    \right) u_1v_2 = 0.
\end{align*}
It follows that $T$ acts on $A$ preserving the $\ZZ^p$-grading.
\end{proof}

We note that the above result still holds in the case that some $q_i=-1$ so long as we require that $g$ act as some $\eta_{\gamma,\mu}$ with $\gamma \neq 1$.

Returning to the case that $A$ is a rank one GWA, our final result finishes the classification in the case that $q=-1$ and $g$ acts as some $\Omega \circ \eta_{\gamma,\mu}$. That is
\begin{align}
\label{eq.gomega}
g(t) = -\gamma t, \quad
g(v) = \mu u, \quad
g(u) = \mu\inv \gamma^{D} v.
\end{align}

\begin{theorem}\label{thm.omega}
Fix a GWA $A = \kk[t](u,v,\sigma,h)$ with defining polynomial $h$ of degree $D>0$, and defining automorphism $\sigma(t) = -t$. Let $T=T_n(\lambda,m)$ be a generalized Taft algebra.
Suppose that $A$ is an inner-faithful weakly $\ZZ$-graded $T$-module algebra where $g$ acts as an automorphism $\Omega \circ \eta_{\gamma,\mu}$ and the action of $x$ is given by \eqref{eq.xact}. Then
\begin{enumerate}
    \item $\lambda=\gamma=-1$,
    \item $\phi=0$,
    \item  $\sigma(\alpha_{11})=\pm \alpha_{11}$, $\alpha_{22}=-\alpha_{11}$, $\alpha_{12}=-\mu\sigma(\alpha_{11})$, $\alpha_{21}=\mu\inv \sigma\inv(\alpha_{11})$,
    and
    \item $\sigma(h)=\gamma^{D}h$.
\end{enumerate}

Conversely, defining an action of $g$ as in \eqref{eq.gomega} and an action of $x$ by \eqref{eq.xact} subject to conditions (1)--(4),
then it is possible to extend this action to all of $A$ to make $A$ an inner-faithful weakly $\ZZ$-graded $T$-module algebra.
\end{theorem}
\begin{proof}
Suppose that $A$ is an inner-faithful weakly $\ZZ$-graded $T$-module algebra where $g$ acts as an automorphism $\Omega \circ \eta_{\gamma,\mu}$.
Since $A$ is a $T$-module algebra, the action of $x$ must preserve the relations of $A$ so
\begin{align}
\label{sp.vuh}
0 &= x(vu-h) 
   = (\mu \sigma(\alpha_{11})+\alpha_{12})u^2 + \mu \sigma(\alpha_{21})\sigma(h) + \alpha_{22}h - x(h),\\
\label{sp.uvh}
\begin{split}
0 &= x(uv-\sigma(h)) 
    = (\mu\inv \gamma^{D}\sigma\inv(\alpha_{22})+\alpha_{21})v^2 + \mu\inv \gamma^{D}\sigma\inv(\alpha_{12})h \\
    &\hspace{10em} + \alpha_{11}\sigma(h) - x(\sigma(h)).
\end{split}
\end{align}
By the $\ZZ$-grading on $A$, we see that $\alpha_{12}=-\mu\sigma(\alpha_{11})$ and $\alpha_{21}=-\mu\inv \gamma^{D}\sigma\inv(\alpha_{22})$. 
Further, since the relations of $T$ must act as $0$ on $A$, we have
\begin{align}
\label{sp.gxu}
0 &= (gx-\lambda xg)(u)
    = -\gamma^{D}(\sigma\inv(\alpha_{22}) - \lambda \sigma(\alpha_{11}) )u + \mu\inv \gamma^{D}(\alpha_{11} - \lambda \alpha_{22})v, \\
\label{sp.gxv}
0 &= (gx-\lambda xg)(v)
    = -\gamma^{D}(\sigma(\alpha_{11}) - \lambda \sigma\inv(\alpha_{22}))v + \mu(\alpha_{22} - \lambda \alpha_{11})u,  \\
\label{sp.gxt}
0 &= (gx-\lambda xg)(t)
    = g(\phi) - \lambda x(-\gamma t)
    = \phi(-\gamma t) + \lambda\gamma \phi(t).
\end{align}
By \eqref{sp.gxu} and \eqref{sp.gxv},
$\alpha_{11}=\lambda\alpha_{22}=\lambda^2\alpha_{11}$,
so either $\lambda=\pm 1$ or else $\alpha_{ij}=0$ for all $i,j$. By Hypothesis~\ref{hyp.taft}, $\lambda \neq 1$.

Further we have
\begin{align}
\label{sp.ut}
0 &= x(ut+tu)
    = (\mu\inv \gamma^{D})v\phi + \phi u - (1+\gamma)t(\alpha_{11}u + \alpha_{21}v), \\
\label{sp.vt}
0 &= x(vt+tv)
    = \mu u\phi + \phi v - (1+\gamma)t(\alpha_{12}u + \alpha_{22}v).
\end{align}
By \eqref{sp.gxu}, \eqref{sp.ut}, and \eqref{sp.vt},
\[
\phi = (1+\gamma)t\alpha_{11} = -(1+\gamma)t\alpha_{22} = - \phi.
\]
Thus, $\phi=0$ and $\gamma=-1$. Note this implies that $g(t)=t$.

By \eqref{sp.uvh},
\[ 0    = \mu\inv \gamma^{D}\sigma\inv(\alpha_{12})h + \alpha_{11}\sigma(h)
    = \alpha_{11}(\sigma(h)-\gamma^{D}h),
\]
so $\sigma(h)=\gamma^{D}h$. 


Now we need
\begin{align*}
0 &= x^2(u) = (\alpha_{11}^2+\alpha_{21}\alpha_{12})u + \alpha_{21}(\alpha_{11}+\alpha_{22})v =  (\alpha_{11}^2+\alpha_{21}\alpha_{12})u, \\
0 &= x^2(v) = \alpha_{12}(\alpha_{11}+\alpha_{22})u + (\alpha_{22}^2+\alpha_{21}\alpha_{12})v =  (\alpha_{22}^2+\alpha_{21}\alpha_{12})v.
\end{align*}
This gives
\begin{align*}
0 &= \alpha_{11}^2+\alpha_{21}\alpha_{12}
    = \alpha_{11}^2+(-\mu\inv \sigma\inv(\alpha_{22})(-\mu \sigma(\alpha_{11}))
    = \sigma\inv (\sigma(\alpha_{11})^2 - \alpha_{11}^2),
\end{align*}
so $\sigma(\alpha_{11})=\pm \alpha_{11}$.

Conversely, let $A$ be a quantum GWA with $q=-1$, let $T=T_2(-1,2)$ be the Sweedler Hopf algebra. Suppose $g$ acts on $A$ as an automorphism $\Omega \circ \eta_{-1,\mu}$ for some $\mu \in \kk^\times$. Suppose $x(t)=0$, $x(u)=\alpha_{11} u + \alpha_{21} v$, and $x(v)=\alpha_{12} u + \alpha_{22} v$ with the $\alpha_{ij} \in \kk[t]$ satisfying (3).
Finally, suppose $\sigma(h)=(-1)^D h$ where $D=\deg_t(h)$. We extend the action of $x$ to all of $A$. It is then easy to check that \eqref{sp.vuh}--\eqref{sp.vt} are satisfied. Thus, this makes $A$ into a $T$-module algebra.
\end{proof}

\section{Fixed rings by Taft actions}
\label{sec.fixed}

For a Hopf algebra $H$ and an $H$-module algebra $A$, the \emph{fixed ring of $A$ by $H$} is
\[ A^H = \{ a \in H \mid h(a)=\epsilon(h)a \text{ for all } h \in H\}.\]
This generalizes the usual definition of the fixed ring of a group action.
In this section we consider fixed rings of quantum GWAs by Taft algebra actions.

Let $A=\kk[t](u,v,\sigma,h)$ be a quantum GWA and set $\eta=\eta_{\gamma,\mu}$.
Suppose $\gcd(\ord(\gamma),\ord(\mu))=1$ and $\ord(\gamma) \mid D$. Then $A^{\grp{\eta}}$ is again a quantum GWA generated by $t^{\ord(\gamma)}$, $u^{\ord(\mu)}$, and $v^{\ord(\mu)}$ \cite[Corollary 3.5]{GHo}. 
However, if $T=T_n(\lambda)$ is a Taft algebra and $A$ an inner-faithful weakly $\ZZ$-graded $T$-module algebra such that $g$ acts as some $\eta_{\gamma,\mu}$ with $\gamma \neq 1$, then $\gcd(\ord(\gamma),\ord(\mu))\neq 1$ in general (in fact this holds if and only if $\mu=1$). Hence, the above method will not apply when computing $A^T = (A^{\grp{g}})^{\grp{x}}$.

\begin{theorem}
\label{thm.fixed}
Fix a GWA $A = \kk[t](u,v,\sigma,h)$ with defining polynomial $h$ of degree $D>0$ and defining automorphism $\sigma(t) = qt$ where $q \in \kk^\times$ is a root of unity, $q \neq 1$. 
Let $T=T_n(\lambda)$ be a Taft algebra. Suppose that $A$ is an inner-faithful weakly $\ZZ$-graded $T$-module algebra where $g$ acts as $\eta_{\gamma,\mu} \in \Aut(A)$ with $\gamma \neq 1$.
\begin{enumerate}
    \item \label{fix1} Suppose $k>0$ and $x(v) \neq 0$. If $q^k \neq 1$, then $(A^T)_{-k}=0$. If $q^k=1$, then
    $f(t)v^k \in (A^T)_{-k}$ if and only $\supp(f(t)) \subset \{ a \in \NN \mid \gamma^a \mu^k = 1\}$. 
    
    \item \label{fix2} Suppose $k>0$ and $x(u) \neq 0$. If $q^k \neq 1$, then $(A^T)_k=0$. If $q^k=1$, then
    $g(t)u^k \in (A^T)_k$ if and only if $\supp(g(t)) \subset \{ b \in \NN \mid \gamma^{b+kD} \mu^{-k} = 1\}$.
    
    \item \label{fix3} $A^T$ is a commutative ring.
    In particular, $A^T \iso \kk[U,V,T]/(UV-H)$ for some polynomial $H \in \kk[T]$.
\end{enumerate}
\end{theorem}
\begin{proof}
 By the definition of $\eta_{\gamma,\mu}$ and Lemma~\ref{lem.special}, the actions of both $g$ and $x$ respect the grading on $A$ and hence $(A^T)_k = (A_k)^T$ for all $k \in \ZZ$. Clearly, $(A_0)^\grp{g} = \kk[t^n]$ and $x(t^n)=0$ by \eqref{eq.xtn}. Thus, $(A_0)^T = \kk[t^n]$. 

\eqref{fix1} By Theorem \ref{thm.main}, $x(v) = \alpha_{22} v$ where $\alpha_{22} = \frac{1-\mu q^{1-d}}{1-\gamma} \phi^{(1)}$. Thus, 
\[ x(v^2) 
    = \left(\frac{1-\mu q^{1-d}}{1-\gamma}\right) \left(\mu \phi^{(1)}(q\inv t) + \phi^{(1)}(t)\right)v^2.
\]
But $q$ is an $n^\text{th}$ root of unity and $\supp(\phi)$ (and hence $\supp(\phi^{(1)})$) is contained in a single congruence class mod $n$. It follows that $\phi^{(1)}(q\inv t) = q^{1-d} \phi^{(1)}(t)$. Set $\omega=\mu q^{1-d}$. Then $x(v^2) = \alpha_{22}[2]_\omega v^2$.
An induction argument then shows that $x(v^k)=\alpha_{22}[k]_\omega v^k$.

Suppose $t^a v^k \in A^T$. Then $t^a v^k=g(t^a v^k)=\gamma^a \mu^k t^a v^k$. Thus, $\gamma^a\mu^k=1$. Now,
\begin{align*}
x(t^av^k)
    &= g(t^a)x(v^k) + x(t^a)v^k \\
    &= \gamma^a t^a ( \alpha_{22} [k]_\omega v^k) + \phi [a]_\gamma t^{a-1}v^k \\
    &= \left( \frac{\gamma^a-\gamma^a\mu^k q^{k(1-d)}}{1-\gamma} + [a]_\gamma \right)\phi^{(1)} t^a v^k \\  
    &= \left( \frac{1-q^{k(1-d)}}{1-\gamma} \right)\phi^{(1)} t^a v^k.
\end{align*}
Hence, $x(t^a v^k)=0$ if and only if $q^{k(1-d)}=1$. Since $\gcd(d-1,n)=1$, then this is equivalent to $q^k=1$.

By linearity, it follows that $f(t)v^k \in A^T$ then $\supp(f(t))$ is contained in a single congruence class modulo $n$. 

\eqref{fix2} This is similar to (2).


\eqref{fix3} First consider the case $\alpha_{11},\alpha_{22} \neq 0$. Let $k$ be the minimal positive integer such that $q^k=1$. Choose $a,b \in \NN$ minimal such that $\gamma^a \mu^k=1$ and $\gamma^{b+kD}\mu^{-k}=1$. By (1)-(3), $A^T$ is generated as an algebra by $T=t^n$, $V=t^av^k$, and $U=t^bu^k$. Clearly $U$ and $V$ commute with $T$ since $q^n=1$. Moreover, induction shows that
\begin{align*}
UV &= t^{a+b} \prod_{i=1}^k \sigma^i(h) =: H \quad\text{and}\quad
VU = t^{a+b} \prod_{i=0}^{k-1} \sigma^{-i}(h).
\end{align*}
Since 
\begin{align}\label{eq.sigma}
\sigma^{-i}(h(t)) = h(q^{-i} t) = q^{-Di} h(t)
    = q^{D(k-i)} h(t)
    = h(q^{k-i} t)
    = \sigma^{k-i}(h(t)),
\end{align}
then it follows that $UV=VU$. Hence, $A^T \iso \kk[U,V,T]/(UV-H)$.

Now suppose $\alpha_{11}=\alpha_{22}=0$. Let $t^av^k \in A^T$, then
\[ 0 = x(t^av^k) = g(t^a)x(v^k) + x(t^a)v^k = [a]_\gamma t^a v^k,\]
so $n \mid a$. Assuming this, we have $g(t^av^k)=\gamma^a \mu^k t^av^k = \mu^k t^av^k$, so $\ord(\mu) \mid k$. It follows that $A^T$ is generated by $T=t^n$, $V=v^k$, and $U=u^k$, where $k=\ord(\mu)$. As above, $U$ and $V$ commute with $T$ since $q^n=1$. The argument in \eqref{eq.sigma} shows that $UV=VU$. Hence, $A^T \iso \kk[U,V,T]/(UV-H)$ as before.
\end{proof}

The main result of Theorem \ref{thm.fixed} is consistent with several other related results in the literature. The fixed ring of a quantum plane or quantum Weyl algebra by a Taft algebra acting linearly and inner-faithfully  was shown to be a commutative polynomial ring \cite[Lemma 2.1]{GWY}. Similar results exist in the case of generalized Taft actions \cite[Proposition 5.6]{CG}. On the other hand, these fixed rings are very different from the fixed ring under actions of cyclic groups, which shows that considering Hopf actions is very significant.

If we consider the associated graded ring (with respect to the standard $\NN$-filtration) of the fixed rings appearing in Theorem~\ref{thm.fixed}, we obtain a Kleinian singularity.
The automorphisms of Kleinian singularities are well-studied \cite{BD,ML2}. In light of Theorem~\ref{thm.fixed}, it would be interesting to further study the relationship between the representation theory of the GWA and the corresponding Kleinian singularity. 

We remark that in the setting of Theorem \ref{thm.omega} (i.e., when $g$ acts via some $\Omega \circ \eta_{\gamma,\mu}$), the fixed ring is not always commutative. In the notation of that theorem, take $D$ to be even and $\sigma(\alpha_{11})=\alpha_{11}$. By \cite[Lemma 3.7]{GHo}, $A^{\grp{g}}$ is generated by $T=t$ and $W=\mu u + v$. In fact, $A^{\grp{g}} \iso \kk_{-1}[W,T]$, the $(-1)$-skew polynomial ring in two variables. Now $x(T)=0$ and one checks that $x(W)=0$ so in fact $A^T \iso \kk_{-1}[W,T]$.

\bibliographystyle{amsplain}

\providecommand{\bysame}{\leavevmode\hbox to3em{\hrulefill}\thinspace}
\providecommand{\MR}{\relax\ifhmode\unskip\space\fi MR }
\providecommand{\MRhref}[2]{%
  \href{http://www.ams.org/mathscinet-getitem?mr=#1}{#2}
}
\providecommand{\href}[2]{#2}

\end{document}